\documentclass[a4paper,12pt,oneside]{article}
\usepackage{hyperref}
\usepackage[T1]{fontenc}

\usepackage{authblk} 
\usepackage{mathtools}
\usepackage{mathrsfs}
\usepackage{amssymb}
\usepackage{xifthen}
\usepackage{tikz}
\usepackage{amsmath}
\usepackage{amsthm}
\usepackage{mathrsfs}
\usepackage{latexsym}
\usepackage[all]{xy}
\usepackage{epsfig}
\usepackage{amscd,mathrsfs,graphicx,mathrsfs}

\usepackage{amsthm,amsmath,amsfonts,amssymb}
\usepackage{xcolor}
   \usepackage{geometry}
\newtheorem{theorem}{Theorem}[section]

\newtheorem{lemma}[theorem]{Lemma}
\newtheorem{corollary}[theorem]{Corollary}
\newtheorem{remark}[theorem]{Remark}
\newtheorem{conjecture}[theorem]{Conjecture}

\newgeometry{vmargin={28mm}, hmargin={20mm,17mm}}   

\makeatletter

\hypersetup{
    colorlinks,
    linkcolor={red},
    citecolor={green},
    urlcolor={blue}
}

\newtheorem*{theorem*}{Theorem}

\usepackage{verbatim}
\usepackage{xcolor}
\usepackage{authblk}
\hypersetup{
    colorlinks,
    linkcolor={orange},
    citecolor={magenta},
    urlcolor={blue}
}

\title{Proof of the Noferini--Williams conjecture for Gilbert--Howie groups}

\author[1]{Ihechukwu Chinyere\thanks{Corresponding author: \texttt{i.chinyere@up.ac.za; ihechukwu@aims.ac.za}}}
\affil[1]{Department of Mathematics and Applied Mathematics, University of Pretoria, Hatfield 0028, Pretoria, South Africa}

\begin{document}
\maketitle

\begin{abstract}
\noindent
The Gilbert–Howie groups $H(n, m)$ form a notable subclass within the broader family of Fibonacci-type cyclically presented groups $G_n(m, k)$. Noferini and Williams conjectured that the abelianization $H(n, m)^{ab}$ is torsion-free with \(\mathbb{Z}\)-rank $2$ if and only if $n \equiv 0 \pmod{6}$ and $m \equiv 2 \pmod{n}$. We confirm this conjecture by proving that \[ \operatorname{Res}(F, G) = \pm 1, \] where \[ F = \frac{1 + t^m - t}{\Phi_6} \quad \text{and} \quad G = \frac{t^n - 1}{\Phi_6}, \] with $\Phi_6$ denoting the sixth cyclotomic polynomial. The proof uses a minimality argument, reducing the general problem to three cases: $m = 2 + \frac{n}{3}$, $m = 2 + \frac{n}{2}$, and $m = 2 + \frac{2n}{3}$.
These cases are then handled using polynomial resultant analysis and field-theoretic methods.
 As a consequence, we complete the classification of all $G_n(m, k)$ that arise as labelled oriented graph groups.

\vspace{1em}

\end{abstract}

\section{Introduction}
Cyclically presented groups are finitely presented groups whose relations are invariant under cyclic permutation of generators. Since their introduction, they have provided a setting uniting combinatorial group theory, number theory, and low-dimensional topology. A key theme is expressing algebraic invariants in terms of arithmetic data from the presentation.

\medskip
A particularly important family is given by the Fibonacci-type cyclically presented
groups
\begin{equation}\label{eq:Gnmk}
G_n(m,k)
=
\langle x_0,\dots,x_{n-1}
\mid
x_i x_{i+m}=x_{i+k}\ (0\le i<n)
\rangle,
\qquad (\text{indices mod } n).
\end{equation}

These groups were first discussed in \cite{MR1634446,MR387419} (see also the survey \cite{MR3010816}). They include the classical Fibonacci groups $F(2,n)=G_n(1,2)$\cite{Conway1965}, the Sieradski groups $S(2,n)=G_n(2,1)$ \cite{MR830041}, and the Gilbert-Howie groups $H(n,m)=G_n(m,1)$ \cite{MR1332862}, as well as many generalisations. These groups have been studied from various perspectives, including finiteness and asphericity \cite{MR2003626,MR2464807,MR1622427,MR1298588,MR1332862,MR2488144}, hyperbolicity \cite{MR4223853,MR4243359,MR4424974}, perfectness \cite{MR1700021,MR2661287} (see also \cite{MR4315549}), and topology \cite{MR1332862}.

\medskip
A basic invariant of a finitely presented group is its abelianization.
For cyclically presented groups, this invariant is especially accessible: abelianization gives an integer circulant relation matrix, so its structure is determined by the Smith normal form of this circulant matrix.
Equivalently, the \(\mathbb{Z}\)-rank of the abelianization, that is, the minimal number of generators of its torsion-free part, is determined by the greatest common divisor of the \emph{representer} polynomial (see Section~\ref{prelim}) and $t^n-1$. For the Fibonacci-type presentation \eqref{eq:Gnmk}, the representer polynomial is
\[
1+t^m-t^k.
\]
The \(\mathbb{Z}\)-rank of $G_n(m,k)^{\mathrm{ab}}$ is then equal to
the degree of $\gcd(1+t^m-t^k,t^n-1)$, that is, the number of common roots of
$1+t^m-t^k$ and $t^n-1$, counted with multiplicity.

\medskip

For a generic choice of parameters $n$ and $m$, the abelianization of $H(n,m)$ is finite or of $\mathbb{Z}$-rank 1. In particular, the cases with $\mathbb{Z}$-rank 2 are exceptional. When the abelianization has $\mathbb{Z}$-rank 2, it is necessary that $m \equiv 2 \pmod n$. This places these groups in close relation with the Sieradski groups, for which $m=2$. This motivates the problem of determining exactly when the abelianization is free and has $\mathbb{Z}$-rank 2. Guided by extensive computations and structural results, Noferini and Williams proposed the following conjecture which characterizes exactly when this phenomenon occurs.

\begin{conjecture}[Noferini--Williams \cite{MR4418964}]\label{conj}
For $n\ge1$ and $0\le m<n$, $H(n,m)^{\mathrm{ab}}\cong\mathbb{Z}^2$
if and only if 
$n\equiv0\pmod6$ and 
$m\equiv2\pmod n.$
\end{conjecture}

The forward implication follows from cyclotomic roots of order six for the trinomial $1+t^m-t$ under the stated congruences (see \cite[Lemma 4.2]{MR4418964}).
The converse is more difficult as one must rule out all other possible cyclotomic coincidences uniformly in $n$ and $m$. This brings us to the purpose of this paper, which is to prove Conjecture~\ref{conj}. 

\medskip
Our method is arithmetic: the \(\mathbb{Z}\)-rank~2, torsion-free condition reduces to $\operatorname{Res}(F,G)=1,$ where
\[
F=\frac{1+t^m-t}{\Phi_6(t)},\quad
G=\frac{t^n-1}{\Phi_6(t)}.
\]
The proof proceeds by a minimality argument. In Section~\ref{tech}, we show that the conjecture holds for the groups $H(n, 2 \pm n/3)$ and $H(n, 2 + n/2)$. We also derive strong necessary conditions on $n$ and $m$ that any minimal counterexample must satisfy. In Section \ref{sec:mainproof}, we show that any such minimal counterexample must in fact belong to one of these families. Consequently, no minimal counterexample exists, and the conjecture follows.

\section{Preliminaries and notations}\label{prelim}

This section lays the foundation for the remainder of the paper by presenting the core ideas and setting the notation that will be used consistently throughout. All notations introduced here will remain in force for the remainder of the article and may not be recalled.

\medskip
For an integer $n\geq 1$, a \emph{cyclic presentation} is a group presentation of the form
\[
P_n(w)=\langle x_0,\dots,x_{n-1}\mid w(x_i,x_{i+1},\dots,x_{i+n-1})\ (0\le i<n)\rangle,
\]
subscripts are taken modulo $n$, and
\[
w=w(x_0,x_1,\dots,x_{n-1})
\]
is a word in the free group on the generators $x_0,\dots,x_{n-1}$. The group defined by the presentation $P_n(w)$ is called the cyclically presented group and is denoted by $G_n(w)$.

\medskip
Let $c_i$ denote the exponent sum of the generator $(x_i)$ in the word $(w)$, for $0 \leq i < n$. The relation matrix of $G_n(w)^{\text{ab}}$ is then the $n \times n$ \emph{circulant matrix} $ C = \operatorname{circ}_n(c_0, c_1, \ldots, c_{n-1})$, where each row is a cyclic shift of the previous one. The group $G_n(w)$ is perfect if and only if $C$ is unimodular. In contrast, $G_n(w)^{\text{ab}}$ is infinite if and only if $C$ is singular. Associated with $C$ is the \emph{representer polynomial}, whose structure reflects the combinatorics of the presentation,
\[
f_C(t)=\sum_{i=0}^{n-1} c_i t^i,
\]
whose structure reflects the combinatorics of the presentation. As shown in \cite[Equation~3.2.14]{MR543191} and \cite[Theorem~3, p.~78]{MR695161}, the determinant of $C$ satisfies
\[
|\det C|
=
|\operatorname{Res}(f_C,g)|
=
\prod_{i=0}^{n-1} |f_C(\zeta^i)|,
\]
where $g(t)=t^n-1$, $\zeta$ is a primitive $n$th root of unity, and
$\operatorname{Res}(f_C,g)$ denotes the resultant of $f_C$ and $g$. Equivalently, grouping the roots of unity according to their orders, we obtain
\[
|\operatorname{Res}(f_C,g)|
=
\prod_{d\mid n} \bigl|\operatorname{Res}(f_C,\Phi_d)\bigr|
=
\prod_{d\mid n}
\Bigl|
\operatorname{N}_{\mathbb{Q}(\zeta_d)/\mathbb{Q}}
\bigl(f(\zeta_d)\bigr)
\Bigr|,
\]
where $\Phi_d$ denotes the $d$th cyclotomic polynomial and
$\operatorname{N}_{\mathbb{Q}(\zeta_d)/\mathbb{Q}}$ is the field norm. Since we will only be concerned with absolute values, throughout this article,
we adopt the convention that $\operatorname{Res}(A,B)$ denotes
$|\operatorname{Res}(A,B)|$.

\medskip
We recall some basic properties of the resultant.
For $A,B,C\in\mathbb{Z}[t]$, we have
\begin{equation}\label{eq:res-shift}
\operatorname{Res}(A+BC,B)=\operatorname{Res}(A,B),
\end{equation}
and
\begin{equation}\label{eq:res-mult}
\operatorname{Res}(A,BC)
=
\operatorname{Res}(A,B)\operatorname{Res}(A,C).
\end{equation}

Moreover, for $p,q\in\mathbb{Z}$ with $p\neq0$ and $P\in\mathbb{Z}[t]$, we have
\begin{equation}\label{eq:res-linear}
|\operatorname{Res}(pt+q,P)|
=
|p|^{\deg P}\,|P(-q/p)|.
\end{equation}

Let $\varphi$ denote Euler’s totient function. For $k\neq l$, the resultant of $\Phi_k$ and
$\Phi_l$ is given by
\begin{equation}\label{eq:res-cyclotomic}
\operatorname{Res}(\Phi_k,\Phi_l)
=
\begin{cases}
p^{\varphi(l)}, & \text{if } k = l p^j \text{ for some prime $p$ and $j\ge 1$},\\[2mm]
p^{\varphi(k)}, & \text{if } l = k p^j \text{ for some prime $p$ and $j\ge 1$},\\[2mm]
1, & \text{otherwise},
\end{cases}
\end{equation}
see \cite{MR251010}.

\medskip
For a prime $p$ and an integer $x$, we write $v_p(x)$ for the $p$-adic valuation of $x$. For $k \ge 0$, we denote by $L_k$ the $k$-th Lucas number defined by
\[
L_k = \phi^k + (-\phi)^{-k},
\]
where $\phi = \frac{1+\sqrt{5}}{2}.$

\subsection{Standing Assumptions}\label{std}
Throughout the remainder of the article, we assume that $2\leq m<n$ and that
$a,b,r,s,u,v\ge1$ are integers such that
\[
m-2=2^b3^r u,\qquad
n=2^a3^s v,\qquad
\gcd(uv,6)=1.
\]

We also define
\[
F(t)=\frac{f(t)}{\Phi_6(t)},
\qquad
G(t)=\frac{g(t)}{\Phi_6(t)},
\]
where
\[
f(t)=t^m-t+1,
\qquad
g(t)=t^n-1.
\]

All lemmas and corollaries in Section~\ref{tech} are stated with terms whose definitions were previously established. Therefore, throughout this section, the meanings of these terms remain unchanged, even as the values of $n$ and $m$ vary with the parameters $a,b,r,s,u,v$; the notation remains consistent.

\section{Technical Results}\label{tech}
In this section, we present the technical results that underpin the proof of Conjecture~\ref{conj}. All notation is as in Section~\ref{std}.

\medskip
Our first result extends \cite[Corollary~5.7]{MR4418964}, which was previously used to verify Conjecture~\ref{conj} in cases where $n=24d$ and $\gcd(6, d)=1$. The result demonstrates that, in certain situations, the size of the torsion part grows at least quadratically in the number of generators.

\begin{lemma}\label{lucas}
Under Assumptions~\ref{std}, assume that $a>b$, and let $K=\gcd(m-2,n)$. Assume further that $m\equiv K+2 \pmod{2K}$. Then $2+L_K \mid \operatorname{Res}(f,G)$. In particular, if $\operatorname{Res}(F,G)=1$, then
\[
2+L_K \mid n^2/3.
\]
\end{lemma}
\begin{proof}
The integer $K$ is divisible by $6$, so $\Phi_6(t)$ divides $t^K-1$. Since $a>b$, $2K$ divides $n$. Hence, the standard factorization (see \cite[\S3.3]{MR1421575}) shows that $t^K+1$ also divides $t^n-1$. We claim that $\gcd(t^K+1,\Phi_6(t))=1$. Indeed, if $\Phi_6(\zeta)=0$, then $\zeta$ is a primitive sixth root of unity, so $\zeta^3=-1$. If $\zeta$ is also a root of $t^K+1$, then $\zeta^K=-1$, so $\zeta^{K-3}=1$, which implies that $K\equiv 3 \pmod{6}$, contradicting $K\equiv 0 \pmod{6}$.

\medskip
It follows by Euclid’s lemma that $t^K+1$ divides $G(t)$, so $G(t)=(t^K+1)Q(t)$ for some $Q(t)\in\mathbb{Z}[t]$. The assumption $m\equiv K+2 \pmod{2K}$ implies that
\[
t^m \equiv t^{K+2}\equiv -t^2 \pmod{t^K+1}.
\]
Therefore
\[
f(t)=t^m-t+1\equiv -t^2-t+1 \pmod{t^K+1}.
\]

Hence, by \eqref{eq:res-mult}, we have
\[
\operatorname{Res}(f,t^K+1)
=\operatorname{Res}(-t^2-t+1,t^K+1)
=\operatorname{Res}(t^2+t-1,t^K+1).
\]
Let $\alpha,\beta$ be the roots of $t^2+t-1$, so $\alpha,\beta\in\{\phi^{-1},-\phi\}$, where $\phi=(1+\sqrt{5})/2$. Hence
\[
\operatorname{Res}(f,t^K+1)
=(\alpha^{K}+1)(\beta^{K}+1)
=2+L_K.
\]

By \eqref{eq:res-mult} again and using $G(t)=(t^K+1)Q(t)$, we have
\[
\operatorname{Res}(f,G)
=\operatorname{Res}(f,(t^K+1)Q)
=\operatorname{Res}(f,t^K+1)\operatorname{Res}(f,Q)
=(2+L_K)\operatorname{Res}(f,Q),
\]
and hence $2+L_K$ divides $\operatorname{Res}(f,G)$. Since $f=\Phi_6 F$, we conclude using \eqref{eq:res-mult} yet again that
\begin{equation}\label{eq:res-factor}
\operatorname{Res}(f,G)
=
\operatorname{Res}(\Phi_6 F,G)
=
\operatorname{Res}(\Phi_6,G)\operatorname{Res}(F,G).
\end{equation}

If $\operatorname{Res}(F,G)=1$, then
\begin{equation}\label{eq:res-factor2}
\operatorname{Res}(f,G)=\operatorname{Res}(\Phi_6,G).
\end{equation}

Let $\zeta$ be a root of $\Phi_6$. Since $6$ divides $n$, we have $\zeta^n=1$. As $\zeta$ and $\overline{\zeta}$ are simple roots of $t^n-1$, we can apply L’Hôpital’s rule at $t=\zeta$. This gives
\[
G(\zeta)=\frac{n\zeta^{\,n-1}}{2\zeta-1}.
\]
Hence,
\[
\operatorname{Res}(\Phi_6, G)
=
G(\zeta)\,G(\overline{\zeta})
=
\frac{n^2}{3},
\]
and combining this with \eqref{eq:res-factor} and \eqref{eq:res-factor2}, we conclude that $2+L_K$ divides $n^2/3$.
\end{proof}

\begin{corollary}\label{bge22-special}
Under Assumptions~\ref{std}, suppose that $n \ge 22$ and $a>b$. If $m = 2 + n/2$, then $\operatorname{Res}(F,G) \neq 1$.
\end{corollary}

\begin{proof}
Let $K = n/2$. Then all the hypotheses of Lemma~\ref{lucas} are satisfied with $n = 2K$ and $m = K+2$. Assume, for contradiction, that $\operatorname{Res}(F,G)=1$. By Lemma~\ref{lucas}, it follows that $2+L_K$ divides $n^2/3$. Since $K = n/2$, we have
\[
2+L_{n/2} \mid \frac{n^2}{3}.
\]
However, an easy induction argument shows that $2+L_{n/2} > n^2/3$ for $n \ge 22$, which is a contradiction. Therefore, $\operatorname{Res}(F,G)\neq 1$.
\end{proof}

\medskip
Our next result shows that if $\operatorname{Res}(F,G)=1$, then $v$ divides $u$.

\begin{lemma}\label{vg1}
Under Assumptions~\ref{std}, if $v\nmid u$, then $\operatorname{Res}(F,G)\neq 1$.
\end{lemma}

\begin{proof}
Suppose for contradiction that $\operatorname{Res}(F,G)=1$ and that $v\nmid u$. Then $v>1$. By \cite[Lemma~4.3]{MR4418964}, one of the following holds:
\[
v\mid m-1,\quad v\mid m-2,\quad v\mid 1.
\]
Since $\gcd(v,m-2)=1$ and $v>1$, the only possibility is $v\mid m-1$. Write $m=1+\alpha v$ for some $\alpha\in\mathbb{N}$. Since $6\mid m-2$, it follows that $6\mid\alpha v- 1$. Set $\alpha_0 = n/v$. Since $\alpha_0$ is even, an argument similar to that used in the proof of Lemma~\ref{lucas} shows that $t^v + 1$ divides $t^n - 1$. Since $\gcd(6, v) = 1$, we have
\[
\gcd(t^v+1,\Phi_6(t))=1.
\]
Hence $t^v+1$ divides $G(t)$, and so for some $S(t)\in\mathbb{Z}[t]$, we may write
\[
G(t)=(t^v+1)S(t).
\]
By \eqref{eq:res-mult}, we have
\[
\operatorname{Res}(F,G)
=\operatorname{Res}(F,t^v+1)\operatorname{Res}(F,S).
\]
Since $\operatorname{Res}(F,G)=1$, it follows that
\begin{equation}\label{eq:resultant-F}
\operatorname{Res}(F,t^v+1)=1.
\end{equation}

As $\alpha$ is odd, we get
\[
t^m=t^{1+\alpha v}=t(t^v)^\alpha\equiv -t \pmod{t^v+1}.
\]
Hence,
\[
f(t)=t^m-t+1\equiv 1-2t \pmod{t^v+1}.
\]
Applying \eqref{eq:res-shift}, we get
\begin{equation}\label{eq:resultant-linear}
\operatorname{Res}(f,t^v+1)
=\operatorname{Res}(1-2t,t^v+1)
=\left|(-2)^{v}\left(2^{-v}+1\right)\right|
=2^{v}+1.
\end{equation}
Since $f = \Phi_6 F$ and $\operatorname{Res}(\Phi_6(t), t^v + 1) = 3$ (as $\gcd(v,6)=1$), it follows that
\begin{equation}\label{eq:resultant-factorization}
\operatorname{Res}(f,t^v+1)
=\operatorname{Res}(\Phi_6,t^v+1)\operatorname{Res}(F,t^v+1)
=3\,\operatorname{Res}(F,t^v+1).
\end{equation}

Therefore, combining \eqref{eq:resultant-linear} and \eqref{eq:resultant-factorization} gives
\[
\operatorname{Res}(F,t^v+1)=\frac{2^{v}+1}{3}.
\]
However, \eqref{eq:resultant-F} forces $v=1$, which contradicts the earlier assumption $v>1$. Therefore, we conclude that if $\operatorname{Res}(F,G)=1$, then $v\mid u$.
\end{proof}

\begin{remark}\label{rems}
Lemma \ref{vg1} implies that under the assumptions~\ref{std}, whenever $\operatorname{Res}(F,G)=1$, we may write $u=\tau v$ for some integer $\tau\ge 1$.
\end{remark}

\begin{lemma}\label{lem:v1_reduction}
Under Assumptions~\ref{std}, assume that $a=v=u=1$, that $b\in\{1,2\}$, and that $r=s-1$ so $(n=2\cdot 3^s$ and $m-2=2^b3^{s-1})$. If $m-2\in \{n/3, 2n/3\}$ and $\operatorname{Res}(F,G)=1$, then $\operatorname{Res}(f,\Phi_n)=9$.
\end{lemma}

\begin{proof}
Since $n=2\cdot 3^s$, we set $K=n/3=2\cdot 3^{s-1}$ so that $(m-2)\in\{K,2K\}$.
Let $d$ be any divisor of $K$ with $d\neq 6$, and let $\zeta$ be a primitive $d$-th root of unity.
Then $\zeta^{K}=1$, hence $\zeta^{m-2}=1$, and therefore
\[
f(\zeta)=\zeta^{m}-\zeta+1=\zeta^2-\zeta+1=\Phi_6(\zeta).
\]
Taking the product over all roots of $\Phi_d$ gives
\begin{equation}\label{eq:res_f_Phi_d}
\operatorname{Res}(f,\Phi_d)=\operatorname{Res}(\Phi_6,\Phi_d).
\end{equation}
Since $f=\Phi_6F$, multiplicativity of the resultant \eqref{eq:res-mult} gives
\[
\operatorname{Res}(f,\Phi_d)
=
\operatorname{Res}(\Phi_6,\Phi_d)\operatorname{Res}(F,\Phi_d).
\]
Combining this with \eqref{eq:res_f_Phi_d}, we obtain $\operatorname{Res}(F,\Phi_d)=1$ for all $d\mid K$, with $d\neq 6$. Since $G(t)=(t^n-1)/\Phi_6(t)$ and
\[
t^n-1=\prod_{d\mid n}\Phi_d(t),
\]
we have
\[
G(t)=\prod_{\substack{d\mid n\\ d\neq 6}}\Phi_d(t).
\]
As $n=2\cdot 3^s$, the divisors of $n$ are exactly the divisors of $K$ together with $3^s$ and $n$.
Hence
\[
G(t)=
\left(\prod_{\substack{d\mid K\\ d\neq 6}}\Phi_d(t)\right)\Phi_{3^s}(t)\Phi_n(t),
\]
and therefore, by \eqref{eq:res-mult},
\begin{equation}\label{eq:res_FG_split}
\operatorname{Res}(F,G)
=
\operatorname{Res}(F,\Phi_{3^s})\,\operatorname{Res}(F,\Phi_n).
\end{equation}

Using again $f=\Phi_6F$ and \eqref{eq:res-mult}, we have
\[
\operatorname{Res}(f,\Phi_{3^s})
=
\operatorname{Res}(\Phi_6,\Phi_{3^s})\operatorname{Res}(F,\Phi_{3^s}),
\qquad \text{and} \qquad
\operatorname{Res}(f,\Phi_n)
=
\operatorname{Res}(\Phi_6,\Phi_n)\operatorname{Res}(F,\Phi_n).
\]
Substituting these into \eqref{eq:res_FG_split} gives
\[
\operatorname{Res}(F,G)
=
\frac{\operatorname{Res}(f,\Phi_{3^s})}{\operatorname{Res}(\Phi_6,\Phi_{3^s})}\cdot
\frac{\operatorname{Res}(f,\Phi_n)}{\operatorname{Res}(\Phi_6,\Phi_n)}.
\]
By \eqref{eq:res-cyclotomic}, we have $\operatorname{Res}(\Phi_6,\Phi_{3^s})=1$ and
$\operatorname{Res}(\Phi_6,\Phi_n)=9$, so
\[
\operatorname{Res}(F,G)=\operatorname{Res}(f,\Phi_{3^s})\cdot \frac{\operatorname{Res}(f,\Phi_n)}{9}.
\]
If $\operatorname{Res}(F,G)=1$, then the right-hand side must equal $1$, which forces
$\operatorname{Res}(f,\Phi_n)=9$.
\end{proof}

Consequently, in order to prove that $\operatorname{Res}(F,G)\neq 1$ in the setting of Lemma~\ref{lem:v1_reduction}, it
suffices to show that $\operatorname{Res}(f,\Phi_n)\neq 9$.

\begin{lemma}\label{grow}
Under Assumptions~\ref{std}, suppose that $a=v=u=1$, that $b\in\{1,2\}$, and that $r=s-1$. Then
\begin{enumerate}
\item $\operatorname{Res}(f,\Phi_n)=9$ if $s=2$.
\item $\operatorname{Res}(f,\Phi_n)>9$ if $s\ge 3$.
\end{enumerate}
\end{lemma}

\begin{proof}
Let $q = 3^{s-1}$ and define $\zeta = \zeta_n$. Set $\eta = \zeta^q$, so that
$\eta$ is a primitive sixth root of unity. Define
$L=\mathbb{Q}(\zeta)$ and $M=\mathbb{Q}(\eta)=\mathbb{Q}(\sqrt{-3})$. Since $n=2\cdot 3^s$, it follows that $\zeta^{n/6}=\eta$, and therefore
$\zeta^{n/3}=\eta^2$. Let $\varepsilon\in\{2,4\}$ be determined by $\zeta^K=\eta^\varepsilon$, where
$K=n/3$ if $\varepsilon=2$, and $K=2n/3$ if $\varepsilon=4$.
Define
\[
P(X)=\eta^\varepsilon X^2-X+1\in\mathbb{Z}[\eta][X],
\]
so that $P(\zeta)=f(\zeta)$. Hence
\[
N_{L/M}(f(\zeta))
=
\operatorname{Res}_X(X^q-\eta,P(X))
=
R_q,
\]
where
\[
R_q=1-\eta S_q+\eta^2
\quad\text{and}\quad
S_q=\beta_1^q+\beta_2^q,
\]
with $\beta_1,\beta_2$ the roots of $P(X)$. Since $M$ is imaginary quadratic, we have
$N_{M/\mathbb{Q}}(x)=|x|^2$ for $x\in M$, and therefore
\[
\operatorname{Res}(f,\Phi_n)
=
N_{M/\mathbb{Q}}(R_q)
=
|R_q|^2.
\]

We now bound $|R_q|$ from below. Since $\eta^2$ is a primitive cube root of
unity, we have $|1+\eta^2|=1$, and hence
\[
|R_q|
=
|-\eta S_q+(1+\eta^2)|
\ge
|S_q|-1.
\]

Moreover, $\beta_1\beta_2=\eta^{2\varepsilon}$ implies
$|\beta_1||\beta_2|=1$, and so
\[
|S_q|
=
|\beta_1^q+\beta_2^q|
\ge
|\beta_1|^q-|\beta_2|^q
=
|\beta_1|^q-|\beta_1|^{-q}.
\]

From the explicit expression for $\beta_1$, one obtains
$|\beta_1|^2>3^{2/3}$, hence $|\beta_1|>3^{1/3}$. Therefore, if $q\ge 9$ then
\[
|\beta_1|^q>(3^{1/3})^9=27,
\qquad
|\beta_1|^{-q}<\frac{1}{27},
\]
so
\[
|S_q|>27-\frac{1}{27}>26.
\]
It follows that
\[
|R_q|\ge |S_q|-1>25,
\]
and hence
\[
\operatorname{Res}(f,\Phi_n)=|R_q|^2>25^2>9.
\]

This proves $\operatorname{Res}(f,\Phi_n)>9$ for $q\ge 9$. For the remaining
case $q=3$, a direct evaluation gives
$\operatorname{Res}(f,\Phi_{18})=9$.
\end{proof}

As a consequence of Lemma~\ref{grow}, we get that $H(n,2\pm n/3)^{ab}$ has nontrivial torsion.

\begin{corollary}\label{special}
Let $n=2\cdot 3^s$ with $s\geq 2$. Then $H(n,2\pm n/3)^{ab}\not\cong \mathbb{Z}^2.$
\end{corollary}
\begin{proof}
This is exactly the situation where $f(t)=t^m-t+1$ with $m=2\pm n/3\pmod n$. Set $K=\pm n/3\pmod n$ and apply Lemma~\ref{lem:v1_reduction} and Lemma~\ref{grow}.
\end{proof}

The following lemma is used to provide, in certain situations, a smaller pair $(n_1, m_1)$ satisfying the same conditions as $(n, m)$. This is particularly useful for minimality arguments where one needs to show that no smaller counterexample exists.

\begin{lemma}\label{minimality}
Under Assumptions~\ref{std}, let $u=\tau v$ for some $\tau\ge 1$. If $a\ge b+2$ or $s\ge r+2$ or $(a=b+1$ and $s=r+1)$, then there exists an integer $n_1$ with $n_1\mid n$ and $n_1<n$, and an integer
$m_1$ such that $2<m_1<n_1$, $m_1\equiv m\pmod{n_1}$, and $m_1\equiv2\pmod6$.
\end{lemma}

\begin{proof}
If $a\ge b+2$, we set $n_1:=2^{b+1}3^s v$. Then $n_1\mid n$ and
$n/n_1=2^{a-b-1}\ge2$, so $n_1<n$.

\medskip
If $s\ge r+2$, we instead set
$n_1:=2^{a}3^{r+1} v$. In this case, $n_1\mid n$ and
$n/n_1=3^{s-r-1}\ge3$, so again $n_1<n$.

\medskip
Finally, if $a=b+1$ and
$s=r+1$, we set $n_1:=2^{b+1}3^{r}v$. Then $n_1\mid n$ and $n/n_1=3$, so $n_1<n$. In all cases, we have $6\mid n_1$. Let $m_1$ be the unique integer with
$0\le m_1<n_1$ and $m_1\equiv m\pmod{n_1}$. Since $6\mid n_1$ and
$m\equiv2\pmod6$, it follows that $m_1\equiv2\pmod6$.

\medskip
Suppose, for contradiction, that $m_1=2$. Then $n_1\mid(m-2)$, and hence $v_2(n_1)\le v_2(m-2)$ and $v_3(n_1)\le v_3(m-2)$. However, by construction
of $n_1$, in each of the three cases, we have either
\[
v_2(n_1)=b+1> b=v_2(m-2)
\]
or
\[
v_3(n_1)=r+1> r=v_3(m-2).
\]
This is a contradiction, hence $m_1\neq2$. Consequently, we have $2<m_1<n_1$.
\end{proof}

\section{Main Result}\label{sec:mainproof}

In this section, we give a proof of Conjecture~1.1. To begin, let us recall the setting. Our aim is to show that $H(n,m)^{ab}$ is isomorphic to $\mathbb{Z}^2$ if and only if $n \equiv 0 \pmod{6}$ and $m \equiv 2 \pmod{n}$. By~ \cite[Lemma 4.2]{MR4418964}, a necessary and sufficient condition for \(\mathbb{Z}\)-rank to be~$2$ is $m \equiv 2 \pmod{6}$ and $n \equiv 0 \pmod{6}$. Hence, we write
\[
n = 2^{a} 3^{s} v \quad \text{and} \quad m = 2^{b} 3^{r} u,
\]
where $\gcd(uv,6)=1$ and $a,s,b,r \geq 1$. 

\medskip
A direct computation shows that
\[
H(18,14)^{\mathrm{ab}} \cong H(18,8)^{\mathrm{ab}} \cong \mathbb{Z}^2 \oplus \mathbb{Z}/19\mathbb{Z}, \qquad
 H(12,8)^{\mathrm{ab}} \cong \mathbb{Z}^2 \oplus \mathbb{Z}/5\mathbb{Z}.
\]
In particular, all three groups have \(\mathbb{Z}\)-rank~2 and nontrivial torsion.
Hence, while the conditions \(n \equiv 0 \pmod{6}\) and \(m \equiv 2 \pmod{6}\) are sufficient for the abelianization to have \(\mathbb{Z}\)-rank~2, they are not sufficient to ensure that it is torsion-free.

\medskip
We use all the results obtained in the previous sections to complete the proof of Conjecture~1.1.

\begin{theorem}
Conjecture \ref{conj} holds.
\end{theorem}
\begin{proof}
By \cite[Corollary 1.2]{MR4418964} it is enough to show that $\operatorname{Res}(F,G)= 1$, where $F,G$ are as given in Assumptions \ref{std}. We assume that $n\ge 24$ since the conjecture is true for $n\leq 18$. 

\medskip
Suppose in contradiction that $\operatorname{Res}(F,G)=1$ and $n\nmid (m-2)$, so $2<m<n$. Then by  Lemma \ref{vg1} $v|u$. Hence, we have $n = 2^{a}3^{s}v$ and $m-2 = 2^{b}3^{r}\tau v,$ where $b, a,s,r,\tau, v\ge 1$ and $\gcd(6,\tau v)=1$.

\medskip
We assume that $(n,m)$ is a \emph{minimal counterexample}, that is, $n$ is minimal with the properties:
\begin{itemize}
\item $6\mid n$ and $6\mid (m-2)$,
\item $n\nmid (m-2)$, and
\item $H(n,m)^{ab}\cong\mathbb Z^2$.
\end{itemize}

By Lemma~\ref{minimality}, we have
$a\le b+1$
and 
$s\le r+1.$
Moreover, the equalities $a=b+1$ and $s=r+1$ are excluded by
Lemma~\ref{minimality}. On the other hand, the inequalities $a\le b$ and $s\le r$
are excluded by our assumption that $n\nmid(m-2)$: indeed, if $a\le b$ and
$s\le r$, then
\[
2^a3^sv \mid 2^b3^r\tau v,
\]
and hence
\[
n=2^a3^s v \mid 2^b3^r\tau v=m-2,
\]
a contradiction. Therefore, exactly one of the following four cases occurs:
\begin{enumerate}
\item $a=b+1,\ s=r;$
\item $a=b,\ s=r+1;$
\item $a=b+1,\ s<r;$
\item $a<b,\ s=r+1.$
\end{enumerate}

\textbf{Case (1): $a=b+1$ and $s=r$.} In this case,
\[
n=2^{b+1}3^s v,
\qquad
m-2=2^b3^s\tau v.
\]
Since $m<n$, we have $m-2<n$, and hence
\[
2^b3^s\tau v < 2^{b+1}3^s v.
\]
Dividing both sides by $2^b3^s v$ gives $\tau<2$, so $\tau=1$. Therefore,
\[
m-2=2^b3^s v=\frac{n}{2},
\]
and
\[
m=2+\frac{n}{2}.
\]
This contradicts Corollary~\ref{bge22-special}.

\medskip
\textbf{Case (2): $a=b$ and $s=r+1$.} Here,
\[
n=2^b3^{r+1}v,
\qquad
m-2=2^b3^r\tau v.
\]
Again, $m<n$ implies
\[
2^b3^r\tau v < 2^b3^{r+1}v,
\]
and dividing both sides by $2^b3^r v$ gives $\tau<3$. Since $\gcd(\tau,6)=1$, we must have
$\tau=1$. Hence,
\[
m-2=2^b3^r v=\frac{n}{3},
\]
so
\[
m=2+\frac{n}{3}.
\]
This contradicts Corollary~\ref{special}.

\medskip
\textbf{Cases (3) and (4).} We now consider the remaining cases $a=b+1,\ s<r$ and $a<b,\ s=r+1$. Recall that
\[
n=2^a3^s v,
\quad
m-2=2^b3^r\tau v,
\quad
\gcd(\tau v,6)=1.
\]

We first show that $v=1$. Suppose that $v>1$. Set
\[
n'=\frac{n}{v},
\qquad
m'\equiv m \pmod{n'}.
\]
Then $n'<n$, $6\mid n'$, and $H(n',m')^{ab}\cong\mathbb Z^2$. Moreover, in Case~(3)
we have $a=b+1$, and in Case~(4) we have $s=r+1$, so in both cases
\[
n'\nmid(m-2).
\]
Hence, $(n',m')$ is a smaller counterexample, contradicting the minimality of $n$. Therefore, $v=1.$

\medskip

From now on,
\[
n=2^a3^s,
\qquad
m-2=2^b3^r\tau,
\qquad
\gcd(\tau,6)=1.
\]
Since $m<n$, we have
\[
2^b3^r\tau < 2^a3^s.
\]

\medskip
\textbf{Case (3): $a=b+1$ and $s<r$.} Then
\[
n=2^{b+1}3^s,
\qquad
m-2=2^b3^r\tau,
\]
and the inequality $m<n$ becomes
\[
3^{r-s}\tau < 2.
\]
Since $r>s$, we have $r-s\ge1$, so $3^{r-s}\ge3$, and $\tau\ge1$. Hence
$3^{r-s}\tau\ge3$, which is impossible. Therefore, Case~(3) cannot occur.

\medskip
\textbf{Case (4): $a<b$ and $s=r+1$.} Here,
\[
n=2^a3^{r+1},
\qquad
m-2=2^b3^r\tau.
\]
From $m<n$ we obtain
\[
2^{b-a}\tau < 3.
\]
Since $a<b$, we have $b-a\ge1$, so $2^{b-a}\ge2$. Moreover, if $\tau\ne1$ then
$\tau\ge5$, since $\gcd(\tau,6)=1$. This gives $2^{b-a}\tau\ge10$, a contradiction.
Therefore $\tau=1$, so $2^{b-a}<3$,  hence $b-a=1$. So $b=a+1$, and
\[
m-2=2^{a+1}3^r=\frac{2n}{3},
\]
so
\[
m=2+\frac{2n}{3},
\]
which contradicts Corollary~\ref{special}.

\medskip
This shows that there are no counter examples to the conjecture.
\end{proof}

\section{Conclusion}
Having proved Conjecture~\ref{conj}, we have now completed the classification of the groups $G_n(m,k)$ that arise as LOG groups, namely, the Sieradski groups $S(2,n)$ with $n\equiv 0\pmod 6$.

\bibliographystyle{plain}
\bibliography{Refs}

\end{document}